\documentclass[12pt]{amsart}
\usepackage{format}

\newcommand \bfG{{\mathbf G}}
\newcommand \bfT{{\mathbf T}}

\newcommand \sfk{{\mathsf k}}
\newcommand \kun{{\mathsf k^{un}}}
\newcommand \bark{{\bar{\mathsf k}}}

\title[Arthur Packets for p-adic Groups]{ Some unipotent Arthur packets for reductive p-adic Groups}

\begin{document}

\maketitle

\begin{abstract}
Let $\sfk$ be a $p$-adic field and let $\mathbf{G}(\sfk)$ be the $\sfk$-points of a connected reductive group, inner to split. The set of Aubert-Zelevinsky duals of the constituents of a tempered L-packet form an Arthur packet for $\mathbf{G}(\sfk)$. In this paper, we give an alternative characterization of such Arthur packets in terms of the wavefront set, proving in some instances a conjecture of Jiang-Liu and Shahidi. Pursuing an analogy with real and complex groups, 
we define some special unions of Arthur packets which we call \emph{weak} Arthur packets and describe their constituents in terms of their Langlands  parameters. 
\end{abstract}

\tableofcontents

\section{Introduction}

Let $\mathsf k$ be a nonarchimedean local field of characteristic $0$ with ring of integers $\mathfrak o$, finite residue field $\mathbb F_q$ of cardinality $q$ and valuation $\mathsf{val}_{\mathsf k}$. Fix an algebraic closure $\bar{\mathsf k}$ of $\mathsf k$ and let $\kun\subset \bar{\mathsf k}$ be the maximal unramified extension of $\mathsf k$ in $\bar{\mathsf k}$. Let $\mathbf{G}$ be a connected reductive algebraic group defined over $\sfk$ and let $\mathbf{G}(\sfk)$ be the group of $\sfk$-rational points. We assume throughout that $\mathbf{G}(\sfk)$ is inner to split. 

Let $W_{\sfk}$ be the Weil group associated to $\sfk$ \cite[(1.4)]{Tate1979}. Then the Weil-Deligne group is the semidirect product \cite[\S 8.3.6]{Deligne1972}
$$W_{\sfk}' = W_{\sfk} \ltimes \CC$$
where $W_{\sfk}$ acts on $\CC$ via
$$w x w^{-1}=\|w\| x,\qquad x\in \mathbb C,\ w\in W_k.$$
Let $G^{\vee}$ denote the complex Langlands dual group associated to $\mathbf{G}$, see \cite[\S2.1]{Borel1979}. 

\begin{definition}\label{def:Langlandsparams}
A \emph{Langlands parameter} is a continuous homomorphism
\begin{equation}
    \phi: W_{\sfk}'\rightarrow G^\vee
\end{equation}
which respects the Jordan decompositions in $W_{\sfk}'$ and $G^\vee$ (\cite[\S8.1]{Borel1979}). 
\end{definition}

To each Langlands parameter $\phi$, one hopes to associate an $L$-packet $\Pi^{\mathsf{Lan}}_{\phi}(\mathbf{G}(\sfk))$ of irreducible admissible $\mathbf{G}(\sfk)$-representations \cite[\S10]{Borel1979} (note: these packets have not been defined in full generalality, see \cite{Arthur2013} or \cite{Kaletha2022} for a discussion of the status of the local Langlands conjectures).

\begin{definition}%[Section 6, \cite{Arthur1989}]
An \emph{Arthur parameter} is a continuous homomorphism
$$\psi: W_{\sfk}' \times \mathrm{SL}(2,\CC) \to G^{\vee}$$
such that
\begin{itemize}
    \item[(i)]  The restriction of $\psi$ to $W_{\sfk}'$ is a tempered Langlands parameter.
       \item[(ii)] The restriction of $\psi$ to $\mathrm{SL}(2,\CC)$ is algebraic.
\end{itemize}
Arthur parameters are considered modulo the $G^{\vee}$-action on the target.
\end{definition}

For each Arthur parameter $\psi$ there is an associated Langlands parameter $\phi_{\psi}: W_{\sfk}' \to G^{\vee}$ defined by the formula
\begin{equation}\label{eq:ArthurLanglands}\phi_{\psi}(w) = \psi(w,\begin{pmatrix}\|w\|^{1/2} & 0\\0 & \|w\|^{-1/2} \end{pmatrix}), \qquad w \in W_{\sfk}',\end{equation}
where $\|\bullet\|$ is the pull-back to $W_{\sfk}'$ of the norm map on $W_{\sfk}$. The local version of Arthur's conjectures can be stated succinctly as follows.

\begin{conj}[Section 6, \cite{Arthur1989}]\label{conj:Arthur}
For each Arthur parameter $\psi$, there is an associated finite set $\Pi_{\psi}^{\mathsf{Art}}(\mathbf{G}(\sfk))$ of irreducible admissible $\mathbf{G}(\sfk)$-representations  called the \emph{Arthur packet} attached to $\psi$. The set $\Pi_{\psi}^{\mathsf{Art}}(\mathbf{G}(\sfk))$ should satisfy various properties, including:
\begin{itemize}
    \item[(i)] $\Pi^{\mathsf{Lan}}_{\phi_{\psi}}(\mathbf{G}(\sfk)) \subseteq \Pi_{\psi}^{\mathsf{Art}}(\mathbf{G}(\sfk))$,
    \item[(ii)] $\Pi_{\psi}^{\mathsf{Art}}(\mathbf{G}(\sfk))$ consists of unitary representations,
    %\item[(iii)] The elements of $\Pi_{\psi}^{\mathsf{Art}}(\mathbf{G}(k))$ are parameterized by irreducible representations of a finite group $A_{\psi}$ which is recovered in a natural fashion from $\psi$.
\end{itemize}
along with several other properties (e.g. stability, endoscopy), see \cite[Section 4]{Arthur1989} or \cite[Chapter 1]{AdamsBarbaschVogan}.
\end{conj}

%Note that if $\psi$ is trivial on the $\mathrm{SL}(2,\CC)$ factor, then $\phi_{\psi} = \psi|_{W_k'}$, a tempered Langlands parameter. In this case, the inclusion $\Pi^{\mathsf{Lan}}_{\phi_{\psi}}(\mathbf{G}(k)) \subseteq \Pi^{\mathsf{Art}}_{\psi}(\mathbf{G}(k))$ is an equality. Thus, tempered $L$-packets are also Arthur packets.

Although a general definition is lacking, the packets $\Pi_{\psi}^{\mathsf{Art}}(\mathbf{G}(\sfk))$ have been defined in various special cases. For quasisplit orthogonal and symplectic groups, Arthur has defined them using endoscopic transfer and has proved that they satisfy all of the conjectured properties \cite[Theorems 1.5.1 and 2.2.1]{Arthur2013}. His constructions were extended in \cite{Mok2015} to quasisplit unitary groups. Explicit definitions, compatible with Arthur's, were proposed by M\oe glin in \cite{Mo1,Mo2}, see also the exposition in \cite{Xu} and the references therein. For split $G_2$, some Arthur packets were constructed in \cite{GanGurevich}. Recently, Cunningham and his collaborators have proposed a `microlocal' approach to defining Arthur packets, see \cite{Cunningham}. 

There is one simple class of Arthur parameters for which a general (i.e. case-free) definition of Arthur packets is available.

\begin{definition}\label{def:unipotentparam}
An Arthur parameter $\psi$ is \emph{basic} if the restriction $\psi|_{W_{\sfk}'}$ is trivial. 
\end{definition}

By the Jacobson-Morozov theorem, there is a natural bijection
\begin{align*}
\{\text{nilpotent adjoint } G^{\vee}\text{-orbits}\} &\xrightarrow{\sim} \{\text{basic Arthur parameters}\}/G^{\vee}\\
\OO^{\vee} &\mapsto \psi_{\OO^{\vee}}.
\end{align*}
For each nilpotent orbit $\OO^{\vee} \subset \fg^{\vee}$, we choose an $\mathfrak{sl}(2)$-triple $(e^{\vee},f^{\vee},h^{\vee})$ and consider the semisimple element $q^{\frac{1}{2}h^{\vee}} \in G^{\vee}$. This element is well-defined modulo conjugation by $G^{\vee}$, and therefore determines an `infinitesimal character', see Section \ref{subsec:LLC}. There is a well-known involution defined by Aubert and Zelevinsky on the set of irreducible admissible $\mathbf{G}(\sfk)$-representations, see Section \ref{subsec:AZ}. We denote this involution by $X \mapsto \mathrm{AZ}(X)$. The Arthur packets attached to basic Arthur parameters can be defined as follows, , see for example \cite[\S7.1]{Arthur2013}.

\begin{definition}
Let $\OO^{\vee}$ be a nilpotent adjoint $G^{\vee}$-orbit and let $\psi_{\OO^{\vee}}$ be the associated basic Arthur parameter. Then $\Pi^{\mathsf{Art}}_{\psi_{\OO^{\vee}}}(\mathbf{G}(\sfk))$ is the set of irreducible $\mathbf{G}(\sfk)$-representations $X$ with unipotent cuspidal support such that
\begin{itemize}
    \item[(i)] The infinitesimal character of $X$ is $q^{\frac{1}{2}h^{\vee}}$.
    \item[(ii)] $\AZ(X)$ is tempered.
\end{itemize}
\end{definition}

The \emph{wavefront set} of an admissible $\mathbf{G}(\sfk)$-representation is a fundamental invariant coming from the Harish-Chandra-Howe local character expansion. In its classical form, the wavefront set $\WF(X)$ of $X$ is a collection of nilpotent $\bfG(\sfk)$-orbits in the Lie algebra $\mathfrak g(\sfk)$, namely the \emph{maximal} nilpotent orbits for which the Fourier transforms of the associated orbital integrals contribute to the local character expansion of the distribution character of $X$, \cite[Theorem 16.2]{HarishChandra1999}. In this paper, we consider two coarser invariants (see Section \ref{s:wave} for the precise definitions). The first of these invariants is the \emph{algebraic wavefront set}, denoted $\hphantom{ }^{\bar{\sfk}}\WF(X)$. This is a collection of nilpotent orbits in $\mathfrak{g}(\bark)$, see for example \cite[p. 1108]{Wald18} (where it is simply referred to as the `wavefront set' of $X$). The second invariant is $\CUWF(X)$, a natural refinement of $^{\bark}\WF(X)$ called the \emph{canonical unramified wavefront set}, defined recently in \cite{okada2021wavefront}. This is a collection of nilpotent orbits $\mathfrak g(\kun)$ (modulo a certain equivalence relation $\sim_A$). The relationship between these three invariants is as follows: the algebraic wavefront set $\hphantom{ }^{\bar{\sfk}}\WF(X)$ is deducible from the usual wavefront set $\WF(X)$ as well as the canonical unramified wavefront set $\CUWF(X)$. It is not known whether the canonical unramified wavefront set is deducible from the usual wavefront set, but we expect this to be the case, see \cite[Section 5.1]{okada2021wavefront} for a careful discussion of this topic.

For real and complex groups, it is possible to define Arthur packets in terms of the wavefront set, see \cite{BarbaschVogan1985} and \cite[Section 27] {AdamsBarbaschVogan}. Thus, it is natural to hope for a similar definition in the $p$-adic setting. Our main result (Theorem \ref{thm:main}) gives an alternative characterization of the Arthur packets attached to basic Arthur parameters in terms of the canonical unramified wavefront set. The most precise version of our result requires two additional ingredients which will be introduced in Section \ref{sec:preliminaries} (the classification of nilpotent orbits in $\fg(\bar{\sfk})$ and Achar duality). Here is a simplified version.

\begin{theorem}\label{thm:mainintro}
Let $\OO^{\vee}$ be a nilpotent adjoint $G^{\vee}$-orbit and let $\psi_{\OO^{\vee}}$ be the associated basic Arthur parameter. Then $\Pi^{\mathsf{Art}}_{\OO^{\vee}}(\mathbf{G}(\sfk))$ is the set of irreducible representations $X$ with unipotent cuspidal support such that
\begin{itemize}
    \item[(i)] The infinitesimal character of $X$ is $q^{\frac{1}{2}h^{\vee}}$.
    \item[(ii)] The canonical unramified wavefront set $\CUWF(X)$ is minimal subject to (i). 
\end{itemize}
\end{theorem}

As a consequence of Theorem \ref{thm:mainintro}, we prove a conjecture of Jiang-Liu and Shahidi in the case of basic Arthur packets (see Corollary \ref{cor:jiang} below).

If we replace $\CUWF(X)$ in condition (ii) with the coarser invariant $^{\bar{\sfk}}\WF(X)$, we get a larger set of representations, which we call a \emph{weak Arthur packet}. We conjecture that this set is a union of Arthur packets (and, in particular, that its constituents are unitary). 

\subsection{Acknowledgments}

The authors would like to thank Kevin McGerty and David Vogan for many helpful conversations. The authors would also like to thank Anne-Marie Aubert, Colette M\oe glin, David Renard, and Maarten Solleveld for their helpful comments and corrections on an earlier draft of this paper. This research was partially supported by the Engineering and Physical Sciences Research Council under grant EP/V046713/1.
The third author was supported by Aker Scholarship.

\section{Preliminaries}\label{sec:preliminaries}

Let $\mathsf{k}$ be a nonarchimedean local field of characteristic $0$ with residue field $\mathbb{F}_q$ of sufficiently large characteristic, ring of integers $\mathfrak{o} \subset \mathsf{k}$, and valuation $\mathsf{val}_{\mathsf{k}}$. Fix an algebraic closure $\bar{\mathsf{k}}$ of $\mathsf{k}$ with Galois group $\Gamma$, and let $\kun \subset \bar{\mathsf{k}}$ be the maximal unramified extension of $\mathsf{k}$ in $\bar{\mathsf{k}}$. 
Let $\mf O$ be the ring of integers of $\kun$.
Let $\mathrm{Frob}$ be the geometric Frobenius element of $\mathrm{Gal}(\kun/\mathsf{k})$, the topological generator which induces the inverse of the automorphism $x\to x^q$ of $\mathbb{F}_q$.

Let $\bfG$ be a connected reductive algebraic group defined over $\sfk$, inner to split, and let $\bfT \subset \mathbf{G}$ be a maximal torus. For any field $F$ containing $\sfk$, we write $\mathbf{G}(F)$, $\mathbf{T}(F)$, etc. for the groups of $F$-rational points. Let $\bfG_{\ad}=\bfG/Z(\bfG)$ denote the adjoint group of $\bfG$.

Write $X^*(\mathbf{T},\bark)$ (resp. $X_*(\mathbf{T},\bark)$) for the lattice of algebraic characters (resp. co-characters) of $\mathbf{T}(\bark)$, and write $\Phi(\mathbf{T},\bark)$ (resp. $\Phi^{\vee}(\mathbf{T},\bark)$) for the set of roots (resp. co-roots). Let
$$\mathcal R=(X^*(\mathbf{T},\bark), \ \Phi(\mathbf{T},\bark),X_*(\mathbf{T},\bark), \ \Phi^\vee(\mathbf{T},\bark), \ \langle \ , \ \rangle)$$
be the root datum corresponding to $\mathbf{G}$, and let $W$ the associated (finite) Weyl group.
Let $G$ be the complex reductive group with the same absolute root data as $\bfG$ and let $\mathbf{G}^\vee$ be the Langlands dual group of $\bfG$, i.e. the connected reductive algebraic group defined and split over $\ZZ$ corresponding to the dual root datum 
$$\mathcal R^\vee=(X_*(\mathbf{T},\bark), \ \Phi^{\vee}(\mathbf{T},\bark),  X^*(\mathbf{T},\bark), \ \Phi(\mathbf{T},\bark), \ \langle \ , \ \rangle).$$
Set $\Omega=X_*(\mathbf{T},\bark)/\ZZ \Phi^\vee(\mathbf{T},\bark)$. The center $Z(\bfG^\vee)$ can be naturally identified with the irreducible characters $\mathsf{Irr} \Omega$, and dually, $\Omega\cong X^*(Z(\bfG^\vee))$. For $\omega\in\Omega$, let $\zeta_\omega$ denote the corresponding irreducible character of $Z(\bfG^\vee)$.

For details regarding the parametrization of inner twists of a group $\bfG(\mathsf k)$, see \cite[\S2]{Vogan1993}, \cite{Kottwitz1984}, \cite[\S2]{Kaletha2016}, or \cite[\S1.3]{ABPS2017} and \cite[\S1]{FengOpdamSolleveld2021}. We only record here that the equivalence classes of inner twists of $\mathbf G$ are parameterized by the Galois cohomology group 
\[H^1(\Gamma, \mathbf G_{\ad})\cong H^1(F,\mathbf G_{\ad}(\kun))\cong\Omega_{\ad}\cong \Irr Z(\bfG^\vee_{\mathsf{sc}}),
\]
where $\bfG^\vee_{\mathsf{sc}}$ is the Langlands dual group of $\bfG_{\ad}$, i.e., the simply connected cover of $\bfG^\vee$, and $F$ denotes the action of $\mathrm{Frob}$ on $\bfG(\kun)$. We identify $\Omega_{\ad}$ with the fundamental group of $\bfG_{\ad}$. The isomorphism above is determined as follows: for a cohomology class $h$ in $H^1(F, \mathbf G_{\ad}(\kun))$, let $z$ be a representative cocycle. Let $u\in \bfG_{\ad}(\kun)$ be the image of $F$ under $z$, and let $\omega$ denote the image of $u$ in $\Omega_{\ad}$. Set $F_\omega=\Ad(u)\circ F$. The corresponding rational structure of $\bfG$ is given by $F_\omega$.
Let $\bfG^\omega$ be the connected reductive group defined over $\sfk$ such that $\bfG(\kun)^{F_\omega}=\bfG^\omega(\mathsf k)$.
%denote the corresponding group in the inner class of the split form. 
%Note that $\bfG^{1} = \bfG$ (where we view $\bfG$ as an algebraic group over $\sfk$ for this equality).

\

If $H$ is a complex reductive group and $x$ is an element of $H$ or $\fh$, we write $H(x)$ for the centralizer of $x$ in $H$, and $A_H(x)$ for the group of connected components of $H(x)$. If $S$ is a subset of $H$ or $\fh$ (or indeed, of $H \cup \fh$), we can similarly define $H(S)$ and $A_H(S)$. We will sometimes write $A(x)$, $A(S)$ when the group $H$ is implicit. 
The subgroups of $H$ of the form $H(x)$ where $x$ is a semisimple element of $H$ are called \emph{pseudo-Levi} subgroups of $H$.

\medskip

Let $\mathcal C(\bfG(\mathsf k))$ be the category of smooth complex $\bfG(\mathsf k)$-representations and let $\Pi(\mathbf{G}(\mathsf k)) \subset \mathcal C(\bfG(\mathsf k))$ be the set of irreducible objects. Let $R(\bfG(\mathsf k))$ denote the Grothendieck group of $\mathcal C(\bfG(\mathsf k))$.

\subsection{Nilpotent orbits}\label{subsec:nilpotent}

Let $\mathcal N$ be the functor which takes a field $F$ containing $\sfk$ to the set of nilpotent elements of $\mf g(F)$.
By `nilpotent' in this context we mean the unstable points (in the sense of GIT) with respect to the adjoint action of $\bfG(F)$, see \cite[Section 2]{debacker}.
For $F$ algebraically closed this coincides with all the usual notions of nilpotence.
Let $\mathcal N_o$ be the functor which takes $F$ to the set of orbits in $\mathcal N(F)$ under the adjoint action of $\bfG(F)$.
When $F$ is $\sfk$ or $\kun$, we view $\mathcal N_o(F)$ as a partially ordered set with respect to the closure ordering in the topology induced by the topology on $F$.
When $F$ is algebraically closed, we view $\mathcal N_o(F)$ as a partially ordered set with respect to the closure ordering in the Zariski topology.
For brevity we will write $\mathcal N(F'/F)$ (resp. $\mathcal N_o(F'/F)$) for $\mathcal N(F\to F')$ (resp. $\mathcal N_o(F\to F')$) where $F\to F'$ is a morphism of fields.
For $(F,F')=(\sfk,\kun)$ (resp. $(\sfk,\bark)$, $(\kun,\bark)$), the map $\mathcal N_o(F'/F)$ is strictly increasing.
We will write $\mathcal N$ for the nilpotent cone of the Lie algebra of $G$ and $\mathcal N_o$ for its $\Ad(G)$ orbits.
In this case we also define $\mathcal N_{o,c}$ (resp. $\mathcal N_{o,\bar c}$) to be the set of all pairs $(\OO,C)$ such that $\OO\in \mathcal N_o$ and $C$ is a conjugacy class in the fundamental group $A(\OO)$ of $\OO$ (resp. Lusztig's canonical quotient $\bar A(\OO)$ of $A(\OO)$, see \cite[Section 5]{Sommers2001}). There is a natural map 
\begin{equation}
    \mf Q:\mathcal N_{o,c}\to\mathcal N_{o,\bar c}, \qquad (\OO,C)\mapsto (\OO,\bar C)
\end{equation}
where $\bar C$ is the image of $C$ in $\bar A(\OO)$ under the natural homomorphism $A(\OO)\twoheadrightarrow \bar A(\OO)$. There are also projection maps $\pr_1: \cN_{o,c} \to \cN_o$, $\pr_1: \cN_{o,\bar c} \to \cN_o$. We will typically write $\mathcal N^\vee$, $\mathcal N^\vee_o, \cN^{\vee}_{o,c}$, and $\cN^{\vee}_{o,\bar c}$ for the sets $\mathcal N$, $\mathcal N_o, \cN_{o,c}$, and $\cN_{o,\bar c}$ associated to the Langlands dual group $G^\vee$. When we wish to emphasise the group we are working with we include it as a superscript e.g. $\mathcal N_o^\bfG(k)$.

Recall the following classical result.

\begin{lemma}[Corollary 3.5, \cite{Pommerening} and Theorem 1.5, \cite{Pommerening2}]\label{lem:Noalgclosed}
    Let $F$ be algebraically closed with good characteristic for $\bfG$.
    Then there is canonical isomorphism of partially ordered sets $\Theta_F:\mathcal N_o(F)\xrightarrow{\sim}\mathcal N_o$.
\end{lemma}

\subsection{Duality for nilpotent orbits}
\label{sec:duality}
Write
\begin{equation}\label{eq:dBV}
d: \cN_0 \to \cN_0^{\vee}, \qquad d: \cN_0^{\vee} \to \cN_0.
\end{equation}
for the \emph{Barbasch-Lusztig-Spaltenstein-Vogan duality maps} (see \cite[Chapter 3]{Spaltenstein} or \cite[Appendix A]{BarbaschVogan1985}). 
Write 
\begin{equation}
    d_S: \cN_{o,c} \twoheadrightarrow \cN^{\vee}_o, \qquad d_S: \cN^{\vee}_{o,c} \twoheadrightarrow \cN_o
\end{equation}
for the duality maps defined by Sommers in \cite[Section 6]{Sommers2001} and 
\begin{equation}
    D: \cN_{o,\bar c} \to \cN^{\vee}_{o,\bar c}, \qquad D: \cN^{\vee}_{o,\bar c} \to \cN_{o,\bar c}
\end{equation}
for the duality maps defined by Achar in (\cite[Section 1]{Acharduality}). 
Since the latter two maps are less well known we give a (slightly non-standard) account of their definitions and basic properties.

The original precursor to the duality map $d$ is the involution on the set of two sided cells of the finite Hecke algebra $\mathcal H$ attached to $G$ with equal parameters apparent in \cite{weylcells}.
Upon composing this involution with the Springer correspondence (and identifying the two Hecke algebras obtained from $G$ and $G^\vee$) one obtains the duality map $d$ \cite[Chapter 3]{Spaltenstein}.
The map $d$ is not an involution, but satisfies $d^3=d$.
The reason is that the 2-sided cells of $\mathcal H$ only biject with $\im(d)$ (the so called special nilpotent orbits) instead of whole of $\mathcal N_o$ and so $d$ is only an involution when restricted to this set.

One drawback of this construction however is that the non-special nilpotent orbits only play a peripheral role.
Here the affine Hecke algebra $\mathcal H_{aff}$ of $G^\vee$ with equal parameters hints at a remedy for the situation.
In particular, the 2-sided cells of $\mathcal H_{aff}$ are in bijection with $\mathcal N_o^\vee$.
Moreover, for any element $s\in T$, the finite Hecke algebra $\mathcal H_s$ associated to connected centraliser $C_G^\circ(s)$ embeds into $\mathcal H_{aff}$.
Along with an induction operation on 2-sided cells defined in \cite[Section 6.5]{cellsiv}, we are now in a position to define $d_S$.
For a pair $(\OO,C)\in \mathcal N_{o,c}$ let $x\in \OO$ and $s$ be a semisimple element of $C_G(x)$ such that $sC_G^\circ(x)$ lies in $C$.
Conjugating $s$ appropriately we may assume that it lies in $T$.
Write $L_s = C_G^\circ(s)$ and $d^{L_s}$ for the duality map $d$ for the complex reductive group $L_s$.
Then
$$d_S(\OO,C) = \Ind d^{L_s}(L_s.x).$$
Here we interpret $d^{L_s}(L_s.x)$ as a 2-sided cell of $\mathcal H_s$ and the resulting induced cell of $\mathcal H_{aff}$ as an elements of $\mathcal N_o^\vee$ using the bijections outlined in this and the preceeding paragraph.
Unravelling these constructions gives the clean construction given in \cite{Sommers2001} using Springer theory, and the resulting map $d_S:\mathcal N_{o,c}\to \mathcal N_o^\vee$ is well-defined and has image the whole of $\mathcal N_o^\vee$.
Moreover, it is clear that $d_S(\OO,1) = d(\OO)$ and so it extends the usual map $d$.

This finally brings us to the definition of $D$.
In defining $d_S$ we obtained a map which is no longer symmetric in target and source.
Achar's duality map rectifies this.
In \cite[Proposition 15]{Sommers2001}, Sommers shows that $d_S$ factors through $\mf Q:\mathcal N_{o,c}\to \mathcal N_{o,\bar c}$, and by \cite[Theorem 1]{Acharduality} we have an embedding
$$\mathcal N_{o,\bar c}\xhookrightarrow{i} \mathcal N_o\times \mathcal N_o^\vee, \quad (\OO,C)\mapsto (\OO,d_S(\OO,C)).$$
There is also of course the corresponding embedding for the dual object
$$\mathcal N^\vee_{o,\bar c}\xhookrightarrow{i^\vee} \mathcal N_o^\vee\times \mathcal N_o$$
and this suggests a very natural candidate for $D$, namely $D=(i^\vee)^{-1}\circ \tilde D \circ i$ where 
$$\tilde D:\mathcal N_o\times \mathcal N_o^\vee \to \mathcal N_o^\vee\times \mathcal N_o, \quad (\OO,\OO^\vee) \mapsto (\OO^\vee,\OO).$$
Indeed, whenever $\tilde D\circ i(\xi) \in \im i^\vee$, this is equivalent to the definition for $D(\xi)$ given by Achar. 
However in general, not all $\xi \in \mathcal N_{o,\bar c}$ enjoy this property.
Let us call $\xi \in \mathcal N_{o,\bar c}$ \emph{special} if $\tilde D\circ i(\xi) \in \im i^\vee$ and let us endow $\mathcal N_o\times \mathcal N_o^\vee$ with the partial order 
$$(\OO_1,\OO_1^\vee)\le (\OO_2,\OO_2^\vee), \text{ if } \OO_1\le \OO_2, \OO_1^\vee\ge\OO_2^\vee$$
(which in turn endows $\mathcal N_{o,\bar c}$ with a partial order via the embedding $i$).
By \cite[Theorem 2.4]{Acharduality}, every $\xi\in \mathcal N_{o,\bar c}$ is dominated by a unique smallest special $\xi'$.
We can then define $D(\xi) = (i^\vee)^{-1}\circ \tilde D\circ i(\xi')$, and it is easily seen that this is equivalent to Achar's definition of $D$.
The resulting map $D:\mathcal N_{o,\bar c}\to \mathcal N_{o,\bar c}^\vee$ is then a duality map in much that same way that $d$ is: it satisfies $D^3 = D$.
It also extends $d_S$ in the sense that $\pr_1\circ D = d_S$.
These results are enough for our purposes, but let us just note that in obvious analogy to the preceding situations one might wonder if $\im D$ is in bijection with the 2-sided cells of some Hecke algebra.
This is to the best of out knowledge still unknown.

\subsection{Nilpotent orbits over $\kun$}
Let $\bfT$ be a maximal $\sfk$-split torus of $\bfG$, $\bfT_1$ be a maximal $\kun$-split torus of $\bfG$ defined over $\sfk$ and containing $\bfT$, and $x_0$ be a special point in $\mathcal A(\bfT_1,\kun)$.
In \cite[Section 2.1.5]{okada2021wavefront} the third-named author constructed a bijection
$$\theta_{x_0,\bfT_1}:\mathcal N_o^{\bfG}(\kun)\xrightarrow{\sim}\mathcal N_{o,c}$$
which should be viewed as an analogoue of the bijection in Lemma \ref{lem:Noalgclosed} for nilpotent orbits over $\kun$.
This bijection enjoys a number of properties summarised in the following theorem.
\begin{theorem}
    \label{lem:paramNoK}
    \cite[Theorem 2.20, Theorem 2.27, Proposition 2.29]{okada2021wavefront}
    The bijection 
    $$\theta_{x_0,\bfT_1}:\mathcal N_o^{\bfG}(\kun)\xrightarrow{\sim}\mathcal N_{o,c}$$
    is natural in $\bfT_1$, equivariant in $x_0$, and makes the following diagram commute:
    \begin{equation}
        \begin{tikzcd}[column sep = large]
            \mathcal N_o^{\bfG}(\kun) \arrow[r,"\theta_{x_0,\bfT_1}"] \arrow[d,"\mathcal N_o(\bark/\kun)",swap] & \mathcal N_{o,c} \arrow[d,"\pr_1"] \\
            \mathcal N_o(\bark) \arrow[r,"\Theta_{\bark}"] & \mathcal N_o.
        \end{tikzcd}
    \end{equation}
\end{theorem}
One important consequence of the equivariance in $x_0$ is that the composition
$$d^{un}:= d_S\circ \theta_{x_0,\bfT_1}:\cN_o(\kun)\to \cN_o^\vee$$
is independent of the choice of $x_0$ \cite[Proposition 2.32]{okada2021wavefront}.
We suppress the $\bfT_1$ from the notation since this choice is implicit from fixing the root data and the dual group.

For $\OO_1,\OO_2\in \mathcal N_o(\kun)$ define $\OO_1\le_A\OO_2$ by
$$\OO_1\le_A \OO_2 \iff \mathcal N_o(\bark/\kun)(\OO_1) \le \mathcal N_o(\bark/\kun)(\OO_2),\text{ and } d^{un}(\OO_1)\ge d^{un}(\OO_2)$$
and let $\sim_A$ denote the equivalence classes of this pre-order.
Recall from the previous section that $\cN_{o,\bar c}$ can be endowed with a partial ordering coming from the embedding $\cN_{o,\bar c}\hookrightarrow \cN_o\times \cN_o^\vee$.
The following proposition parameterises $\cN_o(\kun)/\sim_A$ in terms of $\cN_{o,\bar c}$.
\begin{prop}
    \cite[Theorem 2.33]{okada2021wavefront}
    \label{prop:unramclasses}
    The composition $\mf Q\circ \theta_{x_0,\bfT_1}:\mathcal N_o(\kun)\to \mathcal N_{o,\bar c}$ descends to an isomorphism of partial orders 
    $$\bar\theta:\mathcal N_o(\kun)/\sim_A\to \mathcal N_{o,\bar c}$$
    which does not depend on $x_0$.
\end{prop}

\subsection{Representations with unipotent cuspidal support}\label{subsec:LLC} We briefly recall the classification of irreducible representations with unipotent cuspidal support defined in \cite[\S1.6,\S1.21]{Lu-unip1}. 
If $\mathcal S$ is a subset of $G^\vee$, let
\[Z^1_{G^\vee_{\mathsf{sc}}}(\mathcal S)=\text{ preimage of }G^\vee(\mathcal S)/Z(G^\vee)\text{ under the projection } G^\vee_{\mathsf{sc}}\to G^\vee_{\ad},
\]
and let $A^1(\mathcal S)$ denote the component group of $Z^1_{G^\vee_{\mathsf{sc}}}(\mathcal S)$. 

Write $\Phi(G^\vee)$ for the set of $G^\vee$-orbits (under conjugation) of triples $(s,n,\rho)$ where
\begin{itemize}
    \item $s\in G^\vee$ is semisimple,
    \item $n\in \mathfrak g^\vee$ such that $\operatorname{Ad}(s) n=q n$,
    \item $\rho\in \mathrm{Irr}(A^1_{G^{\vee}}(s,n))$.
    %such that $\rho|_{Z(G^\vee)}$ is a multiple of the identity.
\end{itemize}
Without loss of generality, we may assume that $s\in T^\vee$. Note that $n\in\mathfrak g^\vee$ is necessarily nilpotent. The group $G^\vee(s)$ acts with finitely many orbits on the $q$-eigenspace of $\Ad(s)$
$$\mathfrak g_q^\vee=\{x\in\mathfrak g^\vee\mid \operatorname{Ad}(s) x=qx\}$$
In particular, there is a unique open $G^\vee(s)$-orbit in $\mathfrak g_q^\vee$.

Fix an $\mathfrak{sl}(2)$-triple $\{n^-,h^\vee,n\} \subset \fg^{\vee}$ with $h^\vee\in \mathfrak t^\vee_{\mathbb R}$ and set
$$s_0:=sq^{-\frac{h^\vee}{2}}.$$
Then $\operatorname{Ad}(s_0)n=n$.

A parameter $(s,n,\rho)\in \Phi(G^\vee)$ is called \emph{discrete} if $Z_{G^\vee}(s,n)$ does not contain a nontrivial torus.
A discrete parameter is called \emph{cuspidal} if $(u,\rho)$ is a cuspidal pair in $Z_{G^\vee}(s)$ in the sense of Lusztig.

Let $\Pi^{\mathsf{Lus}}(\bfG^\omega(\mathsf k))$ denote the equivalence classes of irreducible $\bfG^\omega(\mathsf k)$-representations with unipotent cuspidal support and
\[\Pi^{\mathsf{Lus}}(\bfG)=\bigsqcup_{\omega\in\Omega_{\mathsf{ad}}} \Pi^{\mathsf{Lus}}(\bfG^\omega(\mathsf k)).
\]
(In this subsection, $\bfG^1(\mathsf k)$ is the split form.)
The following theorem is a combination of several results, namely \cite[Theorems 7.12, 8.2, 8.3]{KL} for $\bfG$ adjoint and Iwahori-spherical representations, \cite[Corollary 6.5]{Lu-unip1} and \cite[Theorem 10.5]{Lu-unip2} for $\bfG$ adjoint and representations with unipotent cuspidal support, \cite[Theorem 3.5.4]{Re-isogeny} for $\bfG$ arbitrary and Iwahori-spherical representations, and \cite{FengOpdamSolleveld2021,FengOpdam2020,Sol-LLC} for $\bfG$ arbitrary and representations with unipotent cuspidal support. See \cite[\S2.3]{AMSol} for a discussion of the compatibility between these  classifications. 
%Define the subset $\mathrm{Irr}(A(s,n))_0 \subset \mathrm{Irr}(A(s,n))$ as in (\ref{eq:defofIrr0}) (taking $S=\{s,n\}$).

\begin{theorem}[{Deligne-Langlands-Lusztig correspondence}]\label{thm:Langlands} There is a bijection
$$\Phi(G^\vee)\xrightarrow{\sim} \Pi^{\mathsf{Lus}}(\bfG),
%\bigsqcup_{\omega\in\Omega_{\mathrm{ad}}}\Pi^{\mathsf{Lus}}(\bfG^\omega(\mathsf k)),
\qquad (s,n,\rho)\mapsto X(s,n,\rho),$$
such that
\begin{enumerate}
    \item $X(s,n,\rho)$ is tempered if and only if $s_0\in T_c^\vee$ (in particular, $\overline {G^\vee(s)n}=\mathfrak g_q^\vee$),
    \item $X(s,n,\rho)$ is square integrable (modulo the center) if and only if it is tempered and $(s,n,\rho)$ is discrete.
    \item $X(s,n,\rho)$ is supercuspidal if and only if $(s,n,\rho)$ is a cuspidal parameter.
\item $X(s,n,\rho)\in \Pi^{\mathsf{Lus}}(\bfG^\omega(\mathsf k))$ if and only if $\rho|_{Z(G^\vee)}$ is a multiple of $\zeta_\omega$. 
\end{enumerate}
This bijection satisfies several natural desiderata (including formal degrees, equivariance with respect to tensoring by weakly unramified characters), see \cite[Theorem 1]{Sol-LLC} and \cite[Theorem 2]{FengOpdamSolleveld2021}.
\end{theorem}
We denote by $\Pi^{\mathsf{Lus}}_s(\bfG(\mathsf k))$ the set of irreducible $\bfG(\mathsf k)$-representations $X(s,n,\rho)$ for a fixed $s\in T^\vee$.

\subsection{Aubert-Zelevinsky duality}\label{subsec:AZ} There is an involution $\AZ$ on the Grothendieck group $R(\mathbf{G}(\sfk))$ of smooth $\mathbf{G}(\sfk)$-representation, called \emph{Aubert-Zelevinsky duality} \cite[\S1]{Au}. This involution can defined in the following manner. Let $\mathcal Q$ denote the set of parabolic subgroups of $\bfG$ defined over $\mathsf k$ and containing $\mathbf B$. For every $\mathbf Q\in\mathcal Q$, let $i_{\mathbf Q(\mathsf k)}^{\mathbf{G}(\mathsf k)}$ and $\mathsf{r}_{\mathbf Q(\mathsf k)}^{\mathbf{G}(\mathsf k)}$ denote the normalized parabolic induction and normalized Jacquet functors, respectively. Then $\AZ$ is defined by
\begin{equation}
    \widetilde{\mathsf{AZ}}: R(\mathbf{G}(\mathsf{k})) \to R(\mathbf{G}(\mathsf{k})), \qquad \widetilde{\mathsf{AZ}}(X)=\sum_{\mathbf Q\in \mathcal Q} (-1)^{r_{\mathbf Q}} ~i_{\mathbf Q(\mathsf k)}^{\bfG(\mathsf k)}(\mathsf{r}_{\mathbf Q(\mathsf k)}^{\bfG(\mathsf k)}(X)),
\end{equation}
where $r_{\mathbf Q}$ is the semisimple rank of the reductive quotient of $\mathbf Q$. 
If a class $X \in R(\mathbf{G}(\sfk))$ is irreducible and $X$ occurs as a composition factor of a parabolically induced module $i_{\mathbf Q(\mathsf k)}^{\bfG(\mathsf k)}(\sigma)$, where $\sigma$ is a supercuspidal representation, then, by \cite[Corollaire 3.9]{Au} $(-1)^{r_{\mathbf Q}}\widetilde{\mathsf{AZ}}(X)$ is the class of an irreducible representation, which we denote $\AZ(X)$. 

Moreover, it is known that $\AZ$ is an involution on the set of irreducible representations in any Bernstein component, see \cite[\S3]{Au}, also \cite[\S3.2]{BBK}. In particular, it preserves the 
irreducible representations with unipotent cuspidal support, and in fact, it induces an involution on $\Pi^{\mathsf{Lus}}_s(\bfG(\mathsf k))$, for each $s\in W\backslash T^\vee$. In addition, by \cite[Theorem 1.1]{BM1} for Iwahori-spherical representations and \cite[Theorem 1.2.1]{BarCiu13} in general, $\AZ$ maps unitary representations to unitary representations in the category of representations with unipotent supercuspidal support.

\subsection{Wavefront sets}\label{s:wave}
Let $X$ be an admissible smooth representation of $\bfG(\sfk)$ and let $\Theta_X$ be the character of $X$.
Recall that for each nilpotent orbit $\OO\in \mathcal N_o^{\bfG}(\sfk)$ there is an associated distribution $\mu_\OO$ on $C_c^\infty(\mf g(\sfk))$ called the \emph{nilpotent orbital integral} of $\OO$ \cite{rangarao}.
Write $\hat\mu_\OO$ for the Fourier transform of this distribution. By a result of Harish-Chandra
(\cite{HarishChandra1999}), there are complex numbers $c_{\OO}(X) \in \CC$ such that
\begin{equation}\label{eq:localcharacter}\Theta_{X}(\mathrm{exp}(\xi)) = \sum_{\OO} c_{\OO}(X) \hat{\mu}_{\OO}(\xi)\end{equation}
for $\xi \in \fg(\sfk)$ a regular element in a small neighborhood of $0$. The formula (\ref{eq:localcharacter}) is called the \emph{local character expansion} of $\pi$. There are several important invariants which can be extracted from the local character expansion. The \textit{($p$-adic) wavefront set} of $X$ is defined to be
$$\WF(X) := \max\{\OO \mid  c_{\OO}(X)\ne 0\} \subseteq \mathcal N_o(\sfk).$$
The \emph{geometric wavefront set} of $X$ is defined to be
$$^{\bark}\WF(X) := \max \{\mathcal N_o(\bark/\sfk)(\OO) \mid c_{\OO}(X)\ne 0\} \subseteq \mathcal N_o(\bark),$$
see \cite[p. 1108]{Wald18} (warning: in \cite{Wald18}, this invariant is called simply the `wavefront set' of $X$). In \cite[Section 2.2.3]{okada2021wavefront} the third author has introduced a third type of wavefront set called
the \emph{canonical unramified wavefront set}. This is defined to be
$$\CUWF(X) := \max \{[\mathcal N_o(\kun/\sfk)(\OO)] \mid c_{\OO}(X)\ne 0\} \subseteq \mathcal N_o(\kun)/\sim_A.$$

Recall the parameterisations $\Theta_{\bark}:\cN_o(\bark)\to \cN_o$ and $\bar\theta:\cN_o(\kun)/\sim_A\to \cN_{o,\bar c}$.
We define

$$\hphantom{ }^{\bar{\sfk}}\WF(X,\CC) := \Theta_{\bar \sfk}(\hphantom{ }^{\bar{\sfk}}\WF(X)), \quad \CUWF(X,\CC) := \bar \theta(\hphantom{ }\CUWF(X)).$$

We have the following basic relation between the geometric and canonical unramified wavefront sets.

\begin{prop}
    \cite[Theorem 2.37]{okada2021wavefront}
    \label{prop:cuwf}
    If $\CUWF(X)$ is a singleton then 
    $$^{\bark}\WF(X,\CC) = \pr_1(\CUWF(X,\CC)).$$
\end{prop}

For our main result below, we will need a description of $\underline{\WF}(X,\CC)$ in the case when $X$ is an irreducible representation with unipotent cuspidal support and real infinitesimal character. Such a description is contained in the main result of \cite{CMBOunipotent} (see also \cite[Theorem 1.2.1]{CMBOIspherical} for the Iwahori-spherical case). 

\begin{theorem}\label{thm:WFformula} Let $X=X(s,n,\rho)$ be an irreducible $\bfG(\mathsf k)$-representation with real infinitesimal character and let $\AZ(X) = X(s,n',\rho')$. Then
\begin{enumerate}
    \item $\underline{\WF}(X,\CC)$ is a singleton, and
\[\underline{\WF}(X,\CC) = D(\OO^{\vee}_{\AZ(X)},1),\]
where $\OO^{\vee}_{\AZ(X)}$ is the $G^\vee$-orbit of $n'$.
\item Suppose $X=X(q^{\frac 12 h^\vee},n,\rho)$ where $h^\vee$ is the neutral element of a Lie triple attached to a nilpotent orbit $\OO^\vee\subset \mathfrak g^\vee$. Then
\[D(\OO^{\vee}, 1) \leq \hphantom{ } \underline{\WF}(X,\CC).
\]
\end{enumerate}
\end{theorem}

\subsection{Basic Arthur packets}\label{subsec:basic}
To formulate the relation between $\AZ$ and Arthur packets, it will be convenient to recall an alternative formulation of the Langlands classification. A \emph{simplified Langlands parameter} is a continuous homomorphism $\widetilde{\varphi}: W_{\sfk} \times \mathrm{SL}(2,\CC) \to G^{\vee}$ such that $\widetilde{\varphi}(W_k)$ consists of semisimple elements and $\widetilde{\varphi}|_{\mathrm{SL}(2,\CC)}$ is algebraic. A simplified Langlands parameter is tempered if the image of $W_k$ is compact (the idea of replacing the Weil-Deligne group $W_k'$ with $W_k \times \mathrm{SL}(2,\CC)$ in the non-Archimedean case was first proposed by Langlands in \cite[p. 209]{Langlands1979}). Although there are good reasons for preferring the Weil-Deligne group formulation, there is a bijection between simplified and honest Langlands parameters, see \cite[p. 278]{Knapp-Langlands}. Similarly, a \emph{simplified Arthur parameter} is a continuous homomorphism $\widetilde{\psi}: W_k \times \mathrm{SL}(2,\CC)_{\mathsf{Lan}} \times \mathrm{SL}(2,\CC)_{\mathsf{Art}} \to G^{\vee}$ such that the restriction of $\widetilde{\psi}$ to $W_k \times \mathrm{SL}(2,\CC)_{\mathsf{Lan}}$ is a tempered simplified Langlands parameter and the restriction of $\widetilde{\psi}$ to $\mathrm{SL}(2,\CC)_{\mathsf{Art}}$ is algebraic. The bijection between simplified and honest Langlands parameters induces a bijection between simplified and honest Arthur parameters. 

If $\widetilde{\psi}$ is a simplified Arthur parameter, we get a second simplified Arthur parameter $\widetilde\psi^t$ by `flipping' the $\mathrm{SL}(2,\CC)$ factors, \cite[p. 390]{Arthur2013}, i.e.
$$\widetilde\psi^t(w,x,y) :=\widetilde \psi(w,y,x),  \qquad w \in W_k, \ x,y \in \mathrm{SL}(2,\CC).$$
Via the bijection between simplified and honest Arthur parameters, the map $\widetilde{\psi} \mapsto \widetilde{\psi}^t$ induces an involution on honest Arthur parameters, which we also denote by $\psi \mapsto \psi^t$. It is expected that
\begin{equation}\label{eq:AZflip}\mathsf{AZ}(\Pi_{\psi}^{\mathsf{Art}}(\mathbf{G}(\mathsf k)) = \Pi_{\psi^t}^{\mathsf{Art}}(\mathbf{G}(\mathsf k)).\end{equation}
In the case when $\mathbf G$ is orthogonal or symplectic, this duality is discussed in \cite[\S7.1]{Arthur2013}. If the Arthur parameter $\psi$ is trivial on the inertia subgroup $I_k$, then the associated Langlands parameter $\phi_{\psi}$ is unramified. So the representations in the Langlands packet $\Pi_{\phi_{\psi}}^{\mathsf{Lan}}(\mathbf{G}(\mathsf k))$, and hence also in the Arthur packet $\Pi_{\psi}^{\mathsf{Art}}(\mathbf{G}(\mathsf k))$, are of unipotent cuspidal support. 

If $\widetilde\psi$ is trivial on $W_k \times \SL(2,\CC)_{\mathsf{Art}}$, then $\Pi_{\psi}^{\mathsf{Art}}(\mathbf{G}(\mathsf k)) = \Pi_{\phi_{\psi}}^{\mathsf{Lan}}(\mathbf{G}(\mathsf k))$ consists of \emph{tempered} representations with real infinitesimal character. So if $\psi$ is basic (cf. Definition \ref{def:unipotentparam}), the desideratum (\ref{eq:AZflip}) suggests the following definition for the associated Arthur packet $\Pi_{\psi}^{\mathsf{Art}}(\mathbf{G}(\sfk))$, see \cite[(7.1.10)]{Arthur2013}.

\begin{definition}
\label{def:basicpacket}
Let $\OO^{\vee}$ be a nilpotent adjoint $G^{\vee}$-orbit and let $\psi_{\OO^{\vee}}$ be the associated basic Arthur parameter (cf. Definition \ref{def:unipotentparam}). Then 
$$\Pi^{\mathsf{Art}}_{\psi_{\OO^{\vee}}}(\mathbf{G}(\mathsf k))  = \{X \in \Pi^{\mathsf{Lus}}_{q^{\frac{1}{2}h^{\vee}}}(\mathbf{G}(\mathsf k)) \mid \AZ(X) \text{ is tempered}\}.$$
\end{definition}

\section{Main results}\label{sec:main-result}

Fix a nilpotent orbit $\OO^{\vee} \subset \cN^{\vee}$ and choose an $\mathfrak{sl}(2)$-triple $(e^{\vee},f^{\vee},h^{\vee})$ with $e^{\vee} \in \OO^{\vee}$.
Recall that if $X = X(s,n,\rho)$, we write $\OO^{\vee}_X$ for the nilpotent adjoint $G^{\vee}$-orbit of $n$.
The semisimple operator $\ad(h^{\vee})$ induces a Lie algebra grading
$$\fg^{\vee} = \bigoplus_{n \in \ZZ}\fg^{\vee}[n], \qquad \fg^{\vee}[n] := \{x \in \fg^{\vee} \mid [h^{\vee},x] = nx\}.$$
Write $L^{\vee}$ for the connected (Levi) subgroup corresponding to the centralizer $\fg^{\vee}_0$ of $h^{\vee}$. If we set $s := q^{\frac{1}{2}h^{\vee}}$, then
$$L^{\vee} = G^{\vee}(s), \qquad \fg^{\vee}[2] = \fg^{\vee}_q$$
where $G^{\vee}(s)$ and $\fg^{\vee}_q$ are as defined in Section \ref{subsec:LLC}. Note that $L^{\vee}$ acts by conjugation on each $\fg^{\vee}[n]$, and in particular on $\fg^{\vee}[2]$.We will need the following well-known facts, see \cite[Section 4]{Kostant1959} or \cite[Prop 4.2]{Lusztigperverse}.

\begin{lemma}\label{lem:orbitclosure}
The following are true:
\begin{itemize}
\item[(i)] $L^{\vee}e^{\vee}$ is an open subset of $\fg^{\vee}[2]$ (and hence the unique open $L^{\vee}$-orbit therein).
\item[(ii)] $L^{\vee}e = \OO^{\vee} \cap \fg^{\vee}[2]$.
\item[(iii)] $G^{\vee}\fg^{\vee}[2] \subseteq \overline{\OO^{\vee}}$.
\end{itemize}
\end{lemma}

\begin{lemma}
    \label{lem:tempclassification}
    Let $X\in \Pi^{\mathsf{Lus}}_{q^{\frac{1}{2}h^{\vee}}}(\mathbf{G}(\mathsf{k}))$.
    Then $X$ is tempered if and only if $\OO^\vee_X = \OO^\vee$.
\end{lemma}
\begin{proof}
Let $X = X(q^{\frac{1}{2}h^{\vee}},n,\rho)$ for $n \in \fg^{\vee}[2]$ and $\rho \in \mathrm{Irr}(A(q^{\frac{1}{2}h^{\vee}},n))$. Note that $s_0 = 1 \in T_c^{\vee}$. So
\begin{align*}
    X \text{ is tempered} &\iff \overline{L^{\vee}n} = g^{\vee}[2] &&\text{(Theorem \ref{thm:Langlands}(i))}\\
    &\iff n \in L^{\vee}e^{\vee}  &&\text{(Lemma \ref{lem:orbitclosure}(i))}\\
    &\iff n \in \OO^{\vee}  &&\text{(Lemma \ref{lem:orbitclosure}(ii))}\\
    &\iff \OO^{\vee}_X = \OO^{\vee}.
\end{align*}
\end{proof}

\begin{theorem}\label{thm:main}
Let $\OO^{\vee}$ be a nilpotent adjoint $G^{\vee}$-orbit and let $\psi_{\OO^{\vee}}$ be the associated basic Arthur parameter (cf. Definition \ref{def:unipotentparam}). Then
    \begin{equation}\label{eq:Arthurpacket}\Pi^{\mathsf{Art}}_{\psi_{\OO^{\vee}}}(\mathbf{G}(\sfk)) = \{X \in \Pi^{\mathsf{Lus}}_{q^{\frac{1}{2}h^{\vee}}}(\mathbf{G}(\mathsf{k})) \mid \ \CUWF(X,\CC) \leq D(\OO^{\vee},1)\}.\end{equation}
\end{theorem}
\begin{proof}
Let $X\in \Pi^{\mathsf{Lus}}_{q^{\frac{1}{2}h^{\vee}}}(\mathbf{G}(\mathsf{k}))$.
By Theorem \ref{thm:WFformula}, we have the bound
$$D(\OO^\vee,1)\le \CUWF(X,\CC).$$
Thus 
\begin{equation}
    \label{eq:bound}
    \CUWF(X,\CC) \le D(\OO^\vee,1) \iff \CUWF(X,\CC) = D(\OO^\vee,1).
\end{equation}
But by Theorem \ref{thm:WFformula} we also have 
\begin{equation}
    \label{eq:wf}
    \CUWF(X,\CC) = D(\OO^\vee_{\AZ(X)},1).
\end{equation}
Therefore
\begin{align*}
    X\in \Pi^{\mathsf{Art}}_{\OO^{\vee}}(\mathbf{G}(\mathsf{k})) &\iff \AZ(X) \text{ is tempered}  &&\text{(Definition \ref{def:basicpacket})} \\
    &\iff \OO^\vee_{\AZ(X)} = \OO^\vee  &&\text{(Lemma \ref{lem:tempclassification})}\\
    &\iff D(\OO^\vee_{\AZ(X)},1)=D(\OO^\vee,1)  &&\text{(\cite[Lemma 2.1.2]{CMBOIspherical})}\\
    &\iff \CUWF(X,\CC)=D(\OO^\vee,1)  &&\text{(Equation \ref{eq:wf})}\\
    &\iff \CUWF(X,\CC)\le D(\OO^\vee,1)  &&\text{(Equation \ref{eq:bound})}.
\end{align*}
\end{proof}

There is a well-known conjecture regarding the (geometric) wavefront sets of the constituents of an Arthur packet. The following formulation essentially appears in \cite[Conjecture 3.2]{jiangliu} (in \cite{jiangliu} it is stated only for classical groups. The version below is the natural generalization to arbitrary groups).

\begin{conj}[{cf. \cite[Conjecture 3.2]{jiangliu}}]\label{conj:jiangliu}
Let $\psi$ be an Arthur parameter and let $\OO^{\vee}_{\psi}$ be the nilpotent $G^{\vee}$-orbit corresponding to the restriction of $\psi$ to $\mathrm{SL}(2,\CC)$. Then
\begin{itemize}
    \item[(i)] For every $X \in \Pi^{\mathsf{Art}}_{\psi}(\mathbf{G}(\sfk))$, there is a bound
    $$\hphantom{ }^{\bark}\WF(X,\CC) \leq d(\OO^{\vee}_{\psi}).$$
    \item[(ii)] The bound in (i) is achieved for some $X \in \Pi^{\mathsf{Art}}_{\psi}(\mathbf{G}(\sfk))$.
\end{itemize}
\end{conj}

We note that a global version of this conjecture appears in \cite{Shahidi90}. An immediate consequence of Theorem \ref{thm:main} is that Conjecture \ref{conj:jiangliu} holds for basic Arthur packets. In fact, something stronger is true for this class of packets.

\begin{cor}\label{cor:jiang}
Let $\OO^{\vee}$ be a nilpotent adjoint $G^{\vee}$-orbit and let $\psi_{\OO^{\vee}}$ be the associated Arthur parameter. Then for every $X \in \Pi^{\mathsf{Art}}_{\psi_{\OO^{\vee}}}(\mathbf{G}(\sfk))$, there is an equality
$$\hphantom{ }^{\bark}\WF(X,\CC) = d(\OO^{\vee}).$$
\end{cor}
\begin{proof}
    Let $X \in \Pi^{\mathsf{Art}}_{\psi_{\OO^{\vee}}}(\mathbf{G}(\sfk))$.
    By Theorem \ref{thm:main} we have that $\CUWF(X,\CC) = D(\OO^\vee_\psi,1)$.
    Recall from the discussion in Section \ref{sec:duality} that $\pr_1\circ D = d_S$ and that $d_S(\OO^\vee,1) = d(\OO^\vee)$.
    Moreover, by Proposition \ref{prop:cuwf}, we have that $^{\bark}\WF(X,\CC) = \pr_1(\CUWF(X,\CC))$.
    Therefore 
    $$^{\bark}\WF(X,\CC) = \pr_1(\CUWF(X,\CC)) = \pr_1(D(\OO^\vee,1))= d(\OO^\vee).$$
\end{proof}

\subsection{Weak Arthur packets} 

Let $\OO^{\vee}$ be a nilpotent adjoint $G^{\vee}$-orbit and let $\psi_{\OO^{\vee}}$ be the associated Arthur parameter. Theorem \ref{thm:main} gives a characterization of the associated Arthur packet $\Pi^{\mathsf{Art}}_{\OO^\vee}(\bfG(\mathsf k))$ in terms of the \emph{canonical unramified wavefront set} $\underline{\WF}(X)$. If we replace $\underline{\WF}(X)$ with the coarser invariant $\hphantom{ }^{\bar{\sfk}}\WF(X)$, we get a set which we will denote by $\Pi_{\psi_{\OO^{\vee}}}^{\mathsf{Weak}}(\mathbf{G}(\sfk))$
\begin{equation}\label{eq:weakpacket}
\Pi_{\psi_{\OO^{\vee}}}^{\mathsf{Weak}}(\mathbf{G}(\sfk)) := \{X = X(q^{\frac{1}{2}h^{\vee}},n,\rho) \mid \hphantom{ }^{\bar{\sfk}}\WF(X) \leq d(\OO^{\vee})\}.
\end{equation}
Because of the compatibility between $\underline{\WF}(X)$ and $\hphantom{ }^{\bar{\sfk}}\WF(X)$, there is a containment
$$\Pi_{\psi_{\OO^{\vee}}}^{\mathsf{Art}}(\mathbf{G}(\sfk)) \subseteq \Pi_{\psi_{\OO^{\vee}}}^{\mathsf{Weak}}(\mathbf{G}(\sfk)).$$
In contrast to the set $\Pi^{\mathsf{Art}}_{\psi_{\OO^\vee}}(\bfG(\mathsf k))$, the set $\Pi_{\psi_{\OO^{\vee}}}^{\mathsf{Weak}}(\mathbf{G}(\sfk))$ is \emph{not} the $\AZ$-dual of a tempered Arthur packet, nor is it parameterized (in any simple way) by representations of $A(\OO^{\vee})$. So it is not a reasonable candidate for an Arthur packet. By analogy with the case of real reductive groups (see \cite[Chapter 27]{AdamsBarbaschVogan}), we call $\Pi_{\psi_{\OO^{\vee}}}^{\mathsf{Weak}}(\mathbf{G}(\sfk))$ the \emph{weak Arthur packet} attached to $\OO^{\vee}$. 

{Using Theorem \ref{thm:WFformula} and Proposition \ref{prop:cuwf}, it is easy to describe the constituents of $\Pi_{\OO^{\vee}}^{\mathsf{Weak}}(\mathbf{G}(\sfk))$.

\begin{cor}\label{c:weak} 
The weak Arthur packet $\Pi_{\OO^{\vee}}^{\mathsf{Weak}}(\mathbf{G}(\sfk))$ is the set of irreducible representations $\mathsf{AZ}(X(q^{\frac{1}{2}h^{\vee}},n,\rho))$, where $n$ belongs to the special piece (in the sense of \cite{Spaltenstein}) of $\OO^\vee.$
\end{cor}
}

Guided by the case of real reductive groups, we conjecture that $\Pi_{\OO^{\vee}}^{\mathsf{Weak}}(\mathbf{G}(\sfk))$ is a union of Arthur packets. Recall the notation of Section \ref{subsec:basic}.

\begin{conj}\label{conj:weakpacket}
The weak Arthur packet $\Pi_{\psi_{\OO^{\vee}}}^{\mathsf{Weak}}(\mathbf{G}(\sfk))$ is a (possibly non-disjoint) union of several Arthur packets 
\begin{equation}\label{eq:unionpackets}\Pi_{\psi_{\OO^{\vee}}}^{\mathsf{Weak}}(\mathbf{G}(\sfk)) = \bigcup_{i=1}^n \Pi_{\widetilde{\psi}_i}^{\mathsf{Art}}(\mathbf{G}(\sfk)),\end{equation}
corresponding to a collection of simplified Arthur parameters
\begin{equation}\{\widetilde{\psi}_i: W_{\mathsf k} \times \mathrm{SL}(2,\CC)_{\mathsf{Lan}} \times \mathrm{SL}(2,\CC)_{\mathsf{Art}} \to G^{\vee} \mid \ 1 \leq i \leq n\},\end{equation}
including the anti-tempered parameter $\widetilde{\psi}_{\OO^{\vee}}$, but also several others. Each $\widetilde{\psi}_i$ is trivial on the Weil group $W_k$ and thus corresponds to an algebraic homomorphism
$$\widetilde{\psi}_i: \mathrm{SL}(2,\CC)_{\mathrm{Lan}} \times \mathrm{SL}(2,\CC)_{\mathsf{Art}} \to G^{\vee}$$
Via the Jacobson-Morozov theorem, each such $\widetilde{\psi}_i$ can be identified with a pair of nilpotent adjoint $G^{\vee}$-orbits
\begin{equation}\label{eq:commutingorbits}(\OO^{\vee}_{i,\mathsf{Art}},\OO^{\vee}_{i,\mathsf{Lan}}),\end{equation}
and a pair of semisimple elements in a single fixed Cartan subalgebra of $\fg^{\vee}$
\begin{equation}\label{eq:commuting semisimple}(h^{\vee}_{i,\mathsf{Art}}, h^{\vee}_{i,\mathsf{Lan}}).\end{equation}
These elements should satisfy
\begin{equation}\label{eq:inflcharsum}\frac{1}{2}h^{\vee} \in G^{\vee}(\frac{1}{2}h^{\vee}_{i,\mathsf{Art}} + \frac{1}{2}h^{\vee}_{i,\mathsf{Lan}}).\end{equation}
\end{conj}

One might be tempted to guess that the set $\{\widetilde{\psi}_i\}$ of parameters appearing in Conjecture \ref{conj:weakpacket} consists of \emph{all} possible parameters satisfying the infinitesimal character condition (\ref{eq:inflcharsum}). But nothing quite so simple is true. The tempered Arthur packet 
corresponding to the pair $(\OO^{\vee}_{\mathsf{Art}},\OO^{\vee}_{\mathsf{Lan}}) = (0, \OO^{\vee})$ is almost never contained in $\Pi_{\OO^{\vee}}^{\mathsf{Weak}}(\mathbf{G}(\sfk))$ (see Section \ref{sec:example} for a detailed example). Thus, the problem of describing the Arthur packets in $\Pi_{\OO^{\vee}}^{\mathsf{Weak}}(\mathbf{G}(\sfk))$ seems interesting and nontrivial. We note that Conjecture \ref{conj:weakpacket} together with Corollary \ref{c:weak} implies the following conjecture (noting that $\mathsf{AZ}$ preserves unitarity):

\begin{conj}
Let $(q^{\frac{1}{2}h^{\vee}},n,\rho)$ be a Deligne-Langlands-Lusztig parameter such that $n$ belongs to the special piece of $\OO^{\vee}$. Then the irreducible representation $X(q^{\frac 12 h^\vee},n,\rho)$ is unitary.
\end{conj}

\section{Example (split $F_4$)}\label{sec:example}

Let $\mathbf{G}(\sfk)$ be the split form of the (unique) simple exceptional group of type $F_4$, and let $\OO^{\vee}=F_4(a_3)$, the minimal distinguished orbit. There are 20 representations $X_1,...,X_{20}$ in  $\Pi^{\mathsf{Lan}}_{q^{\frac{1}{2}h^{\vee}}}(\mathbf{G}(\sfk))$, enumerated in \cite[Section 5]{CiDirac}. All but one of these representations (namely, the representation labeled $X_5$) are Iwahori-spherical. For each $X_k = X(q^{\frac{1}{2}h^{\vee}},n,\rho)$ we record in Table \ref{table:f4a3} below the nilpotent orbit $G^{\vee}n \subset \cN^{\vee}$ and the representation $\rho$ of $A(s,n)$ (in all cases, $A(s,n)$ is a symmetric group, so $\rho$ corresponds to a partition). We also record $\AZ(X_k)$ and indicate whether $X_k$ is unitary. For these calculations, we refer the reader to \cite[Proposition 5.7]{CiDirac}. Using (i) of Theorem \ref{thm:Langlands}, it is easy to determine which representations are tempered (there are 5 such representations, labeled $X_1$,...,$X_5$). Finally, we compute the canonical unramified wavefront set $\CUWF(X_k)$. 
We have by Theorem \ref{thm:WFformula} that
\[\CUWF(\mathsf{AZ}(X_k)) = D(G^\vee n,1).\]
The right hand side can be computed using the tables in \cite[Section 6]{Acharduality}. 

\vspace{5mm}
%\begin{table}%[H]
\begin{longtable}{|c|c|c|c|c|c|c|c|}
%\small
   % \begin{tabular}{|c|c|c|c|c|c|c|c|} 
    \hline
        & $G^{\vee}n$ & $\rho$ & I-spherical? & Tempered? & Unitary? & $\mathsf{AZ}$ & $\CUWF$\\ \hline
        
        $X_1$ & $F_4(a_3)$ & $(4)$ & yes & yes & yes & $X_{20}$ & $(F_4,1)$ \\ \hline
        $X_2$ & $F_4(a_3)$ & $(31)$ & yes & yes & yes & $X_{19}$ & $(F_4(a_1),(12))$\\ \hline
        $X_3$ & $F_4(a_3)$ & $(2^2)$ & yes & yes & yes & $X_{17}$ & $(F_4(a_1),1)$\\ \hline
        $X_4$ & $F_4(a_3)$ & $(21^2)$ & yes & yes & yes & $X_{13}$ & $(C_3,1)$\\ \hline
        
        $X_5$ & $F_4(a_3)$ & $(1^4)$ & no & yes & yes & $X_5$ & $(F_4(a_3), 1)$ \\ \hline

        $X_6$ & $C_3(a_1)$ & $(2)$ & yes & no & yes & $X_{15}$ & $(F_4(a_2),1)$\\ \hline
        $X_7$ & $C_3(a_1)$ & $(1^2)$ & yes & no & yes & $X_9$ & $(F_4(a_3),(1234)$\\ \hline
        
        $X_8$ & $A_1+\widetilde{A}_2$ & $(1)$ & yes & no & yes & $X_8$ & $(F_4(a_3),(123))$\\ \hline
        
        $X_9$ & $\widetilde{A}_1+A_2$ & $(1)$ & yes & no & yes & $X_7$ & $(F_4(a_3),(12))$\\ \hline
        
        $X_{10}$ & $B_2$ & $(2)$ & yes & no & yes  & $X_{18}$ & $(F_4(a_1),1)$\\ \hline
        $X_{11}$ & $B_2$ & $(1^2)$ & yes & no & yes & $X_{11}$ & $(F_4(a_3),(12)(34))$\\ \hline
        
        $X_{12}$ & $A_2$ & $(2)$ & yes & no & yes & $X_{14}$ & $(B_3,1)$\\ \hline
        $X_{13}$ & $A_2$ & $(1^2)$ & yes & no & yes & $X_4$ & $(F_4(a_3),1)$\\ \hline
        
        $X_{14}$ & $\widetilde{A}_2$ & $(1)$ & yes & no & yes & $X_{12}$ & $(C_3,1)$\\ \hline
        
        $X_{15}$ & $A_1+\widetilde{A}_1$ & $(1)$ & yes & no & yes & $X_6$ & $(F_4(a_3),(12))$\\ \hline
        $X_{16}$ & $A_1+\widetilde{A}_1$ & $(1)$ & yes & no & no & $X_{16}$ & $(F_4(a_2),1)$ \\ \hline
        
        $X_{17}$ & $\widetilde{A}_1$ & $(2)$ & yes & no & yes & $X_3$ & $(F_4(a_3),1)$ \\ \hline
        $X_{18}$ & $\widetilde{A}_1$ & $(1^2)$ & yes & no & yes & $X_{10}$ & $(F_4(a_3),(12)(34))$\\ \hline
        
        $X_{19}$ & $A_1$ & $(1)$ & yes & no & yes & $X_2$ & $(F_4(a_3),1)$ \\ \hline
        
        $X_{20}$ & $0$ & $(1)$ & yes & no & yes & $X_1$  & $(F_4(a_3),1)$\\ \hline
   % \end{tabular}
    \caption{Irreducible representations of split $F_4$ with infinitesimal character $q^{h^{\vee}/2}$ for $\OO^{\vee} = F_4(a_3)$.}
    \label{table:f4a3}
%\end{table}
\end{longtable}
\vspace{5mm}
Note that $D(\OO^{\vee},1) = (F_4(a_3),1)$, see \cite[Section 6]{Acharduality}. So the Arthur packet attached to $\OO^{\vee}$ is the set
\begin{equation}\label{eq:smallset}
\Pi_{\OO^{\vee}}^{\mathsf{Art}}(\mathbf{G}(\mathsf{k})) = \{X_5,X_{13},X_{17},X_{19},X_{20}\}\end{equation}
Note that this is precisely the set of anti-tempered representations in $\Pi^{\mathrm{Lus}}_{q^{\frac{1}{2}h^{\vee}}}(\mathbf{G}(\sfk))$, as predicted by Theorem \ref{thm:main}. 

On the other hand, the \emph{weak Arthur packet} attached to $\OO^{\vee}$, see (\ref{eq:weakpacket}), is the much larger set
\begin{equation}\label{eq:bigset}
\Pi^{\mathsf{Weak}}_{\OO^{\vee}}(\mathbf{G}(\sfk)) = \{X_5,,X_7,X_8,X_9,X_{11}, X_{13},X_{15},X_{17},X_{18},X_{19},X_{20}\}\end{equation}
In this case, we can attempt to verify Conjecture \ref{conj:weakpacket} by hand.

There are 10 simplified Arthur parameters $\widetilde{\psi}: \mathrm{SL}(2,\CC)_{\mathsf{Lan}} \times \mathrm{SL}(2,\CC)_{\mathsf{Art}} \to G^{\vee}$ satisfying the infinitesimal character condition (\ref{eq:inflcharsum}). The corresponding pairs $(\OO_{\mathsf{Lan}}^\vee,\OO_{\mathsf{Art}}^\vee)$ of nilpotent orbits are
\[(0,F_4(a_3)), \ (A_1,C_3(a_1)), \ (\widetilde A_1, B_2), \ (A_1+\widetilde A_1,  A_1+\widetilde A_2), \ (\widetilde A_1+A_2,\widetilde A_1+A_2).
\]
together with the `flips' of these pairs. We believe that the Arthur packets contained in the weak Arthur packet (\ref{eq:bigset}) are as in Table \ref{ta:table2}. This is our naive guess for the decomposition of the weak Arthur packet; it is compatible with (\ref{eq:AZflip}), Conjecture \ref{conj:Arthur}(i),(ii), and the expectation for the occurrence of the supercuspidal representation in the Arthur packets.

\begin{table}[H]
    \begin{tabular}{|c|c|c|} \hline
    $\OO_{\mathsf{Lan}}^\vee$ &$\OO_{\mathsf{Art}}^\vee$ &Arthur packet\\
    \hline
    $0$ &$F_4(a_3)$ &\{$X_5, X_{13}, X_{17}, X_{19}, X_{20}\}$\\
    \hline
    $\widetilde A_1$ &$B_2$ &$\{X_5,X_{17}, X_{18}, X_{11}\}$\\
    \hline
    $A_1+\widetilde A_1$  &$A_1+\widetilde A_2$  &$\{X_5,X_{15}, X_8\}$\\
    \hline
    $\widetilde A_1+A_2$  &$\widetilde A_1+A_2$ & $\{X_5,X_9, X_7\}$\\ 
    \hline
\end{tabular}
\caption{Arthur packets in weak packet attached to $F_4(a_3)$.}
\label{ta:table2}
\end{table}

\begin{sloppypar} \printbibliography[title={References}] \end{sloppypar}

\end{document}